\documentclass[11pt]{amsart}
\usepackage[T1]{fontenc}
\usepackage[utf8]{inputenc}
\usepackage[english]{babel}
\usepackage{amsmath, amsthm, amssymb, amsfonts, enumerate, mathrsfs, dsfont, comment, bm, mathtools}
\usepackage{enumitem}
\usepackage[colorlinks=true,linkcolor=blue,urlcolor=blue]{hyperref}
\usepackage{subcaption}
\usepackage{color, geometry, todonotes, indentfirst, graphicx, pdfpages}
\usepackage[title]{appendix}
\usepackage{fancyhdr}
\usepackage{polski}
\usepackage{booktabs}

\newtheorem{theorem}{Theorem}[section]
\newtheorem{remark}[theorem]{Remark}
\newtheorem{assumption}[theorem]{Assumption}
\newtheorem{lemma}[theorem]{Lemma}

\newtheorem{definition}[theorem]{Definition}
\newtheorem{example}[theorem]{Example}
\theoremstyle{plain}
\def \R{\mathbb{R}}

\def \and{\quad \text{and} \quad}

\DeclareMathOperator*{\argmin}{arg\,min}
\DeclareMathOperator*{\argmax}{arg\,max}

\title[]{Multiple equilibria in mean-field game models for large oligopolies with strategic complementarities}  
\author[Dianetti]{Jodi Dianetti}
\author[Federico]{Salvatore Federico}
\author[Ferrari]{Giorgio Ferrari}
\author[Floccari]{Giuseppe Floccari}

\address{J.~Dianetti: Center for Mathematical Economics (IMW), Bielefeld University, Bielefeld, Germany}
\email{\href{mailto:giorgio.ferrari@uni-bielefeld.de}{jodi.dianetti@uni-bielefeld.de}}

\address{S.~Federico: Department of Mathematics, University of Bologna, Bologna, Italy}
\email{\href{mailto:giorgio.ferrari@uni-bielefeld.de}{s.federico@unibo.it}}

\address{G.~Ferrari: Center for Mathematical Economics (IMW), Bielefeld University, Germany}
\email{\href{mailto:giorgio.ferrari@uni-bielefeld.de}{giorgio.ferrari@uni-bielefeld.de}}

\address{G.~Floccari: Tilburg University, Department of Finance, and Bank of Italy, Rome, Italy. }
\email{\href{mailto:giuseppe.floccari@bancaditalia.it}{giuseppe.floccari@bancaditalia.it}}

\date{\today}

\numberwithin{equation}{section}

\begin{document}
\maketitle 
 
\begin{abstract}
We consider continuous-time mean-field stochastic games with strategic complementarities. The interaction between the representative productive firm and the population of rivals comes through the price at which the produced good is sold and the intensity of interaction is measured by a so-called "strenght parameter" $\xi$. Via lattice-theoretic arguments we first prove existence of equilibria and provide comparative statics results when varying $\xi$. A careful numerical study based on iterative schemes converging to suitable maximal and minimal equilibria allows then to study in relevant financial examples how the emergence of  multiple equilibria is related to the strenght of the strategic interaction.
\end{abstract}

\smallskip
{\textbf{Keywords}}: Mean-field games; supermodular cost function; optimal production; strategic complementarity; Tarski's fixed point theorem; fictitious play. 

\smallskip 
{\textbf{AMS subject classification}}: 
93E20, 91A15, 
60H30, 60H10. 

\smallskip 
{\textbf{JEL subject classification}}:
C61, C62, C73, D24, L11.

\section{Introduction}

We investigate the presence of multiple equilibria and multiple equilibrium prices as well as their comparative statics in dynamics oligopolies with strategic complementarities.
Following \cite{Weintraubetal2011}, we use mean-field game (MFG) models to approximate
systems with a large number of firms competing in a market by producing and selling certain goods.
Firms can choose investment plans in order to dynamically expand their own production capacity.
By doing so, they affect the price of their own good, as well as the price of the goods produced by the other firms. 
In particular, we focus on the case in which goods are complementary; i.e., goods whose use is related to the use of associated or paired goods, so that the consumption of one good will increase the  consumption of others. 
We model this interdependence among goods with a specific parameter, which also measures the size of such interaction and will be therefore referred to as \emph{strength parameter}.

Due to the presence of strategic complementarities, the considered model falls into the class of \emph{supermodular MFGs}.
Such class of MFGs has received increased attention in the past few years and several papers study the existence, structure and approximation of equilibria in quite general frameworks (see e.g.\ \cite{dianetti2022strong, dianetti.ferrari.fischer.nendel.2019}).
More in general, in the field of game theory, Topkis \cite{To} initially introduced the submodularity property for static $N$-player games. 
Such a property holds significance in economic and financial applications, as highlighted in  \cite{Vives01} and its references.
From the mathematical point of view, the main consequence of the supermodularity is that it leads to increasing best response maps, which in turn allows to analyse equilibria using lattice theoretical fixed point theorems.

Despite the generality of \cite{dianetti2022strong, dianetti.ferrari.fischer.nendel.2019}, several natural questions which are relevant for the model under consideration are left open.  
Indeed, the presence of strategic complementarities may give rise to multiple equilibria.
Thus one would be interested in knowing the structure of the set of equilibria, comparative statics for the set of equilibria and computational methods for (some of the) equilibria. 

The main contribution of this paper is twofold. 
On the one hand, we present some comparative statics result for equilibria and equilibrium prices in terms of the strength parameter. 
In particular, we show that the minimal and maximal equilibria  increase when the strength parameter increases. 
Moreover, for sufficiently high values of the strength parameter, also the minimal and maximal equilibrium prices manifest the same monotonicity. 
On the other hand, we investigate numerically the presence of multiple equilibria varying the strength parameter. 
In particular, while for low and high values of the strength parameter, only one equilibrium can be found, for intermediate values multiple equilibria arise. 
This is in line with the following interpretation. 
If the interaction is small, the individual optimization prevails over the interaction, thus companies tends to follow closely their own individual optimal investments plans. 
For intermediate levels of the strength parameter, the individual optimization and the interaction balance each other, and firms can coordinate at different equilibria. 
For high values of the interaction, the equilibrium is forced to be unique, which can be intuitively be explained by the fact that the fixed point of a monotone map with sufficiently high slope is necessary unique (see \cite{mou.zhang.2022master.antimon}).
This experiments are interesting also from the theoretical point of view, as they illustrate the convergence of the numerical schemes related to the Banach iteration without necessarily relying on a contraction theorem.

\subsection{Related literature on MFGs}
Mean-field games (MFGs, in short) have been proposed independently by \cite{HuangMalhameCaines06} and \cite{LasryLions07} and arise as limit models for non-cooperative symmetric $N$-player games with interaction of mean field type as the number of players N tends to infinity.
We refer to \cite{Weintraubetal2011} for a detailed discussion on mean-field games in dynamic oligopolies.
While existence of equilibria has been shown in very general frameworks (see \cite{CarmonaDelarue18}),
 the general structure of equilibria and their multiplicity is a relevant issue in MFG theory (as well as in game theory).
Indeed, uniqueness of the equilibrium is known only for short time-horizon and under the well known (but restrictive) Lasry-Lions monotonicity condition (see \cite{CarmonaDelarue18}) and several works discuss multiple equilibria in MFGs (see, e.g., \cite{BardiFischer18, bayraktar.zhang.2020non, Cecchin&DaiPra&Fischer&Pelino19, DelarueFT19} among others).

The challenge of devising algorithms capable of approximating mean-field game solutions was initially tackled in \cite{CardaliaguetHadikhanloo17}. In this work, the authors explored the convergence of fictitious play in potential MFGs without common noise using Partial Differential Equation (PDE) methods. Subsequent investigations delved into variations of this algorithm, examining MFGs with stopping and absorbing boundary conditions in \cite{dumitrescu.leutscher.tankiv.2022linear}, and employing machine learning techniques in \cite{elie2019approximate, perrin2020fictitious, xie2020provable}.
An alternative learning approach for approximating MFG equilibria is outlined in \cite{dianetti.ferrari.fischer.nendel.2019, dianetti.ferrari.fischer.nendel.2022unifying}, specifically for submodular MFGs. 
This method involves iterating the best reply map and exhibits promise, particularly when combined with reinforcement learning techniques, as demonstrated in the recent works \cite{guo.hu.xu.zhamg.2019learning, lee.Rengarajan.Kalathil.Shakkottai2021}.

Within the realm of MFGs, the property of supermodularity (or submodularity for minimization problems) has been previously applied. For instance, in \cite{Adlakhaetal} it was exploited for a class of stationary discrete time games. In \cite{Wiecek17}, the property was harnessed for a class of finite-state MFGs with exit, while in \cite{CarmonaDelarueLacker16} it found application in optimal timing MFGs. 
Moreover, \cite{dianetti.ferrari.fischer.nendel.2019} applied the submodularity property to MFGs involving one-dimensional It\^o-diffusions. 
The recent  \cite{dianetti.ferrari.fischer.nendel.2022unifying} highlights the versatility of the submodularity property in handling qualitatively distinct formulations of MFGs, encompassing scenarios with singular controls, optimal stopping, reflecting boundary conditions, and finite-state problems.

\subsection{Organization of the paper} 
The rest of the paper is organized as follows. 
Section \ref{section A general mean field game model and its properties} presents
the general mean-field game model and discusses some examples, the properties of equilibria, and some approximation results. 
In Section \ref{section numerical analysis} we illustrate some numerical experiments for particular instances of the model and we draw economical interpretations of our findings.
Concluding remarks are outlined in Section \ref{section Concluding discussion}, while  all the proofs are presented in Appendix \ref{Appendix}.

\section{A general mean-field game model and its properties}
    \label{section A general mean field game model and its properties}

In this section we introduce and analyse a model describing large economies of producers of complementary goods. 
Following \cite{Weintraubetal2011}, we use mean-field game (MFG) models to approximate
systems of with large number of firms competing in a market.

\subsection{Model formulation}
Consider a filtered probability space $(\Omega, \mathcal F, \mathbb F = (\mathcal F_t)_{t \in [0,T]}, \mathbb P)$ satisfying the usual conditions of completeness and right-continuity. 
Assume to be given a standard $\mathbb F$-Brownian motion $W$ and  an $\mathcal F_0$-measurable square integrable random variable $x_0$ with distribution $\mu_0$. 

The representative producer of a good chooses an investment plan $\alpha$ in order to expand its production. 
Given $A \subseteq \R$, investment plans are elements of the set 
$$
\mathcal A := \{\alpha : \Omega \times [0,T] \to A \, | \, \text{$\alpha$ is a square-integrable $\mathbb F$-progressively measurable process} \}.
$$
Given the investment plan $\alpha$, the resulting production of a firm is described by a stochastic process $X$ whose evolution is determined by the stochastic differential equation (SDE, in short):
\begin{equation}
    \label{eq:SDE}
dX_t = b(X_t,\alpha_t) dt + \sigma(X_t) dW_t, \quad X_0 = x_0.  
\end{equation}
The
coefficients $ b:\mathbb R \times A \to \mathbb R$ and $\sigma: \mathbb R \to \mathbb R_+$ are continuous and satisfy (with no further reference) the following condition.
\begin{assumption}
$b$ and $ \sigma$ are such that, for any $\alpha \in \mathcal A$, there exists a unique strong solution $X$ to \eqref{eq:SDE}. 
\end{assumption}
The revenues of the representative firm depend on the production level of the entire population, which is described by  a function $\mu: [0,T] \to \mathcal{P}(\R)$, where  $\mathcal{P}(\R)$ denotes the space of probability measures on $\R$.
Given such distribution,
when choosing an investment plan $\alpha$, the revenue of the representative firm at time $t$ is given by $$X_t P(X_t, m_t; \xi),$$
where
\begin{equation}\label{def:m}
m_t:=\int_\R x\,\mu_t(dx).
\end{equation}
Here, 
$P$ is  a price function  to be specified later containing dependence on an exogenous parameter $\xi$ in some interval $E \subset \mathbb R$.  

On the other hand, the investment into production is costly and the cost of investing $\alpha _t$ at time $t$ is given by $c(\alpha_t)$, for a convex function $c: A \to \mathbb R$.

Thus, assuming that revenues and costs are discounted at a constant rate $\rho >0$, for any given $\mu:[0,T]\to \mathcal{P}(\mathbb{R})$ and setting $m$ as in \eqref{def:m}, the representative producer maximizes the net profit functional
\begin{equation*}  
 J(\alpha,m;\xi):= \mathbb{E} \bigg[ \int_0^T e^{-\rho t} \big( X_t P( X_t, m_t; \xi))  - c( \alpha_t) \big) dt + e^{-\rho T} X_T P( X_T, m_T; \xi) \bigg], \quad \alpha \in \mathcal A.
\end{equation*}

In our subsequent examples, the parameter $\xi$ measures the size of the interaction among goods, and will be therefore referred to as the \emph{strength parameter}. 
See the next subsection and Remark \ref{remark LL} for further details.

We introduce the key structural assumption on the model, that will be standing throughout the paper.
\begin{assumption}\label{eq assumption supermodularity}
\begin{itemize}
\item[]
\item[(i)] $\sigma$ is affine;
\item[(ii)] Either $b$ is affine or $b$ is concave and $xP(x,m;\xi)$ is nondecreasing in $x$;
\item[(iii)] $P \in C^2(\mathbb{R}^3)$; 
\item[(iv)] 
 $\partial_{xx} (xP(x,m;\xi)) = 2 \partial_x P(x,m;\xi) + x \partial_{xx} P (x,m;\xi) \leq 0 \ \ \ \forall (x,m) \in \mathbb R^2;$
\item[(v)]
$\partial_x \partial_m (xP(x,m;\xi))  = \partial_m P(x,m;\xi) + x \partial_x \partial_m P (x,m;\xi) \geq 0 ,  \ \ \ \ 
\forall (x,m) \in \mathbb R^2.$ 
\end{itemize}
\end{assumption}
While conditions (i), (ii) and (iv) above ensure standard concavity property for the maximization problem, condition (v) implies that a marginal increase of $m$ increases the marginal profits of the representative firm.
Given $m$ defined as in \eqref{def:m}, an investment $\alpha ^m \in \mathcal A$ is said to be optimal for $m$ if 
$$J(\alpha^m, m) \geq J(\alpha, m) \ \ \ \ \forall  \alpha \in \mathcal A.$$ 
When such an investment exists, one refers to a solution $X^m$ to the SDE as to an optimal production (starting at $x_0$), and to the couple $(X^m,\alpha^m)$ as to an optimal pair for $J(\cdot,m;\xi)$.

\begin{definition}\label{def equilibrium}
Given the initial distribution $\mu_0$ of the population, 
a mean field game equilibrium (MFGE, in short) is a measurable function $m: [0,T] \to \R$ such that  
$$
m_0:=\int_\R x \mu_0(dx) \ \ \mbox{and} \ \ 
m_t = \mathbb E [ X_t^m ], \quad \forall t\in [0,T], 
$$
for the optimal production $X^m$.
\end{definition}
In the sequel, we will address the existence of equilibria, the comparative statics of the equilibria with respect to the strength parameter $\xi$, and computational methods allowing to find numerically the equilibria.

\subsection{Benchmark examples}\label{subsection examples}
We now discuss some natural specifications of the model.

\begin{example}[Mean reverting dynamics and linear demand function]
    \label{Example Mean reverting dynamics and linear demand function}
The first case we consider is the one in which the log-productivity $X$ has a mean reverting dynamics
$$
dX_t = (\alpha_t - \delta X_t) dt + \sigma dW_t, \quad X_0 = x_0 \in \mathbb R,  
$$ 
with a linear inverse demand (as in \cite{Back&Paulsen}, at the end of Section 1)
\begin{align*}
    P( x, m ) & = D  + \xi m - x, 
\end{align*}
for $D>0$ and a parameter $\xi \geq 0$ describing the intensity of the effect of the mean production $m$ on the good of the representative player. 
We further assume a quadratic cost function
$$c(a) = \dfrac{1}{2} a ^2.$$

This particular model falls into the class of linear quadratic MFG problems, which is already addressed by a number of works (see e.g.\ \cite{Bensoussan16} or the book \cite{CarmonaDelarue18} and the references therein). 
\end{example}

\begin{example}
    \label{example hold mr log geo}
We next discuss an example with positive price function. 
We assume that, under no investment, $\log X_t$ follows a controlled  mean reverting dynamics : 
\begin{equation*}
    d \log X_t = \big( - \dfrac{\sigma^2}{2} -\delta \log X_t \big) dt + \sigma dW_t, \quad \log X_0=\log x_0,
\end{equation*}
so that the dynamics of the unaffected state $X_t$ in levels is given by
\begin{equation*}
    d X_t = -  \delta X_t \log X_t dt + \sigma X_t dW_t, \quad X_0 =x_0.
\end{equation*}
Taking $A  = [0,\bar a]$ with $\bar a >0$, the effect of an investment $\alpha$ on the state $X_t$ is assumed to be additive, so that $X$ evolves as
\begin{equation*}
    d X _t = ( \alpha_t- \delta  X_t \log X_t) dt + \sigma X_t dW_t, \quad X_0 =x_0.
\end{equation*}
Further, we consider the case of isoelastic inverse demand, which implies price functions of type  
\begin{equation}\label{eq isoelastic inverse demand}
P(x,m) = D (\gamma m)^{\xi} x^{-\zeta},
\end{equation}
for parameters $D, \gamma>0, \, \xi \geq0$ and $\zeta \in (0, 1)$ and we take a quadratic cost function
$$c(a) = \dfrac{1}{2} a ^2.$$
For this model, one can consider the equilibrium condition as in Definition \ref{def equilibrium}.
\end{example}

\begin{example}[Mean reverting log-dynamics and isoelastic inverse demand]
    \label{Example Mean reverting log-dynamics and isoelastic inverse demand}
Now want to discuss a similar model as the one in Example \ref{example hold mr log geo}, but with a different notion of equilibrium.

Following \cite{lacker.Zariphopoulou.2019mean}, we want to rewrite the optimization problem  in terms of the state variable $Y$, with $Y_t = \log X_t$.
First of all, we want a transformation for the controls.
If $c$ is a process such that $\alpha_t = c_t X_t \in [0,\bar a]$
then we have, for $Y_t = \log X_t$, 
$$
d Y_t = (c_t - \delta Y_t -\sigma^2/2) dt  + \sigma dW_t, \quad Y_0 =  \log(x_0),
$$
so that $\alpha_t \leq \bar a$ becomes
$$
\log c_t + Y_t \leq \log \bar a, 
$$
which is the new admissibility condition.
Thus,  define the set $\mathcal C$ as
$$
\mathcal C := \{ c:[0,T] \to [0,\infty) \, | \, \log c_t + Y_t \leq \log \bar a \},
$$
and, for $\nu_t = \log m_t$, we can rewrite the optimization problem as 
\begin{equation*}
\begin{aligned}
   &   Q(c,\nu):= \mathbb{E} \bigg[ \int_0^T e^{-\rho t} \big(  D e^{ \xi \log \gamma +  \xi \nu_t + (1-\zeta) {Y_t} } - \dfrac{1}{2} (c_t {Y_t}) ^2 \big) dt +  e^{-\rho T}   D e^{  \xi \log \gamma + \xi \nu_T + (1-\zeta) {Y_T} }  \bigg], \\ 
   & \text{subject to}  \quad d Y_t = ( c_t - \delta Y_t -\sigma^2/2)  dt + \sigma  dW_t, \quad Y_0 = y_0 :=   \log(x_0).
\end{aligned}
\end{equation*}
For this problem, we can enforce the equilibrium condition $$\nu_t = \mathbb E [ Y_t^\nu], \ \ \ \forall t\in[0,T].$$ 
In terms of the original $m$ and $X$, this condition becomes
\begin{equation}
    \label{eq def MFG equilibrium in the geometric case}
    m_t = \exp( \mathbb E [ \log X_t^m] ), \ \ \ \ \forall t \in [0,T].
\end{equation}
We refer to the Remark \ref{remark equilibrium type} below for further discussion on this different notion of equilibrium.
\end{example}

\begin{example}
    \label{Example hold GBM and isoelastic inverse demand}
Similarly to the previous example, we assume $A =[0, \bar a ]$, a quadratic cost $c(a) = \frac12 a^2$ and an isoelastic inverse demand function giving a price function $P$ as in \eqref{eq isoelastic inverse demand}.
However, in this case we assume  that $X^\alpha_t$ follows a controlled geometric Brownian motion process
\begin{equation*}
    d X_t = (\alpha_t - \delta X_t ) dt + \sigma X_t dW_t, \quad X_0 =x_0.
\end{equation*}
For this model, one can consider the equilibrium condition as in Definition \ref{def equilibrium}.
\end{example}

\begin{example}[Geometric dynamics and isoelastic inverse demand]
    \label{Example Geometric dynamics and isoelastic inverse demand}
We next consider the model of Example \eqref{Example hold GBM and isoelastic inverse demand}, but with the equilibrium condition for $m$ as in \eqref{eq def MFG equilibrium in the geometric case}.
We give a reformulation in terms of the state variable $Y$. 
Indeed, for $Y_t = \log X_t$ and $\nu_t = \log m_t$, we can rewrite the optimization problem as 
\begin{equation*}
\begin{aligned}
   &  Q(c,\nu):= \mathbb{E} \bigg[ \int_0^T e^{-\rho t} \big(  D e^{ \xi \log \gamma +  \xi \nu_t + (1-\zeta) {Y_t} } - \dfrac{1}{2} (c_t {Y_t}) ^2 \big) dt +  e^{-\rho T}   D e^{  \xi \log \gamma + \xi \nu_T + (1-\zeta) {Y_T} }  \bigg], \\
   & \text{subject to}  \quad d Y_t = ( c_t - \delta -\sigma^2/2)  dt + \sigma  dW_t, \quad Y_0 = y_0 :=   \log(x_0), 
\end{aligned}
\end{equation*}
where the set of controls is defined as
$$
\mathcal C := \{ c:[0,T] \to [0,\infty) \, | \, \log c_t + Y_t \leq \log \bar a \}.
$$
Again,  by enforcing the equilibrium condition $\nu_t = \mathbb E [ Y_t^\nu]$, we obtain and equilibrium condition for $m$ as in \eqref{eq def MFG equilibrium in the geometric case} (see also Remark \ref{remark equilibrium type} below).
\end{example}

\begin{remark}
    \label{remark equilibrium type}
The type of equilibrium used in \eqref{eq def MFG equilibrium in the geometric case} is different from Definition \ref{def equilibrium}, and their difference becomes clear in terms of the associated $N$-player game.
In particular, in the models with isoelastic inverse demand, consider $N$-firms whose productivity is given by $X^1,...,X^N$. 
Denote by $P^i$ the price function for the good produced by fimr $i$.
On the one hand, the equilibrium as in Definition \ref{def equilibrium} is the one arising if in the $N$-player game we would take a price function as
$$
P^i ( X^1,...,X^N ;\xi) = D \Big( \gamma \frac{1}{N}\sum _{j=1}^N X^j_t \Big)^\xi (X_t^i)^{-\zeta}.
$$
On the other hand, the equilibrium as  in \eqref{eq def MFG equilibrium in the geometric case} corresponds to the case with price function
$$
P^i ( X^1,...,X^N ;\xi) =  D \Big( \gamma \Big( \prod _{j=1}^N X^j_t \Big)^{1/N} \Big)^\xi (X_t^i)^{-\zeta}.
$$
We underline that our approach allows to treat both equilibrium types.
Moreover, since the two models give qualitatively similar numerical results, we will discuss only the numerics for Examples \ref{Example Mean reverting log-dynamics and isoelastic inverse demand} and \ref{Example Geometric dynamics and isoelastic inverse demand} (see Section \ref{section Concluding discussion} below).
\end{remark}

\begin{remark}
    \label{remark LL}
    The previous examples illustrate more explicitly the role of the parameter $\xi$.
    First of all, notice that in all of the examples we assume $\xi \geq 0$, which is related to the third condition in Assumption \eqref{eq assumption supermodularity}.
    Secondly, the higher the value of $\xi$ the higher is the relevance of the term $m$ in the payoff of the representative player.
    In this sense, $\xi$ measures the strength of the interaction among players.
    Thirdly, it is important to underline that  negative values of $\xi$  describe the case in which goods are substitute; i.e., goods whose use is related to the use of associated or paired goods, but in which the consumption of one good  decreases the consumption of others. 
    In MFG theory, negative values of $\xi$ are related to the Lasry-Lions monotonicity condition (see \cite{CarmonaDelarue18}), which in turn would imply the uniqueness of the equilibrium.
\end{remark}

\subsection{Existence of MFGE and comparative statics.}
We begin by stating an existence result for the individual player's optimal investment, given an overall mean production level of the firms described by a measurable function $m:[0,T] \to \mathbb R$.
\begin{lemma}\label{lemma optimal controls MFG}
For each $m:[0,T]\to\R$ measurable,  there exists a unique optimal response  $$\alpha^{m} := \argmax_{\alpha \in \mathcal A} J (\alpha, m).$$
\end{lemma}
Since the optimization problem of the individual firm is concave in $\alpha$, the proof of Lemma \ref{lemma optimal controls MFG} follows standard concavity arguments (see Theorem 5.2 at p.\ 68 in \cite{Yong&Zhou99}), and it is therefore omitted.   

The crucial point in our analysis is  that the game exhibits strategic complementarities; that is, 
increments of the opponents strategies incentivize increments in the player optimal strategy. 
This statement is made rigorous in the following lemma. 
\begin{lemma}\label{lemma brm increasing}
Consider two mean productions paths $m$ and $\bar m$ and define  $X^m$ and  $X^{\bar m}$ as the related optimally controlled trajectories. 
If $m_t \leq \bar m_t$ for any $t \in [0,T]$, then $X^m_t \leq X^{\bar m}_t$ for any $t \in [0,T]$, $\mathbb P$-a.s.
\end{lemma}

Under Assumption \ref{eq assumption supermodularity} only, non existence of equilibria is a known issue, as it is discussed in Section 7 in \cite{Lacker15}.
Intuitively, the reason of this is that the representative player's optimal answer might be larger and larger when $m$ becomes large.
Thus, in order to prevent the equilibria to become arbitrarily large, we assume that 
\begin{equation}
    \label{assumption a priori estimates}
    \text{$A = [0,\bar a]$ with $\bar a >0$.}
\end{equation}
This conditions allows us to prove the following result.
\begin{theorem}\label{theorem existence MFGE}
    Under Assumptions \eqref{eq assumption supermodularity} and \eqref{assumption a priori estimates}, there exist a minimal and a maximal MFGE $\underline m$ and $\overline m$; that is,  $\underline m _t \leq m_t \leq \overline m_t$ for any $t \in [0,T]$, for any other MFGE $m$. 
\end{theorem}

\subsection{Comparative statics at equilibrium}
In light of Theorem \ref{theorem existence MFGE}, one can find natural sufficient conditions in order to determine the  equilibrium providing the maximal reward.
\begin{theorem}
    \label{theorem comparative statics first at equilibrium} 
Let $\xi \in E$ and let $\underline m ^\xi, \overline m ^\xi$  the related minimal and maximal MFGE, with associated optimal pairs $(\underline \alpha ^\xi,\underline X ^\xi)$ and $ (\overline \alpha ^{ \xi},\overline X ^\xi)$.
Under Assumptions \eqref{eq assumption supermodularity}, \eqref{assumption a priori estimates}, if the function $m \mapsto xP(x,m,\xi)$ is nondecreasing for all $x\in\R$, then we have a corresponding monotonicity of optimal rewards at equilibrium; that is, 
    $$ 
    \begin{aligned}
    & J(\underline \alpha ^\xi, \underline m ^\xi;\xi) \leq J(\overline \alpha ^\xi, \overline m ^\xi;\xi).
    \end{aligned}
    $$
\end{theorem}
We next address the comparative statics  of the equilibria with respect to the parameter $\xi$, varying
in the set of parameters $E \subset \R$. 
In order to do that, we introduce the following assumption.
\begin{assumption}\label{eq assumption supermodularity parameter}
We have
\begin{equation}
 \partial_x \partial_\xi (xP(x,m;\xi))  = \partial_\xi P(x,m;\xi) + x \partial_x \partial_\xi P (x,m;\xi) \geq 0 ,  \ \ \ \forall (x,m;\xi) \in \mathbb R^2 \times E.
\end{equation}

\end{assumption}
\begin{theorem}\label{theorem comparative statics at equilibrium} 
Take $\xi, \bar \xi \in E$ and let $\underline m ^\xi, \overline m ^\xi$ and $\underline m ^{\bar \xi}, \overline m ^{\bar \xi}$ the related minimal and maximal MFGE, with associated optimal pairs $(\underline \alpha ^\xi,\underline X ^\xi), (\overline \alpha ^{ \xi},\overline X ^\xi), (\underline \alpha ^{\bar \xi}, \underline X ^{\bar \xi}),$ and $(\overline \alpha ^{\bar \xi}, \overline X ^{\bar \xi})$.

Under Assumption
\ref{eq assumption supermodularity parameter}, if $\xi \leq \bar \xi$, then: 
\begin{itemize}
    \item[(i)] $\underline m ^\xi \leq \underline m ^{\bar \xi}$ and $\overline m ^\xi \leq \overline m ^{\bar \xi}$;
    \item[(ii)] $\underline X^\xi \leq \underline X ^{\bar \xi}$ and $\overline X^\xi \leq \overline X^{\bar \xi}$,   $\mathbb P$-a.s.;
\item[(iii)]\label{theorem comparative statics at equilibrium.claim 3} If, further, $(m,\xi) \mapsto xP(x,m,\xi)$ is nondecreasing in the variables $(m,\xi)$, then we have monotonicity of optimal rewards at equilibrium; that is, 
    $$ 
    \begin{aligned}
    & J(\underline \alpha ^\xi, \underline m ^\xi;\xi) \leq J(\overline \alpha ^\xi, \overline m ^\xi;\xi) \leq J ( \overline \alpha ^{\bar \xi}, \overline m ^{\bar \xi}; \bar \xi), \\ 
    & J(\underline \alpha ^\xi, \underline m ^\xi;\xi) \leq J(\underline \alpha ^{\bar \xi}, \underline m ^{\bar \xi};{\bar \xi}) \leq J ( \overline \alpha ^{\bar \xi}, \overline m ^{\bar \xi}; \bar \xi).
    \end{aligned}
    $$
\end{itemize}
\end{theorem}
In the next two subsection, we discuss comparative statics for the examples of Subsection \ref{subsection examples}. Indeed, these examples require some extra arguments, since the assumptions on the monotonicity of $xP(x,m,\xi)$ in $m$ or $\xi$  are not directly satisfied. 

\subsubsection{Comparative statics for Example \ref{Example Mean reverting dynamics and linear demand function}}
We begin by observing that, since the production levels are Gaussian, the production can become negative and we cannot expect monotonicity of $xP(x,m,\xi)$ in $m$ or $\xi$.
However, simple observations will allows us to recover some comparative statics.

Notice that, in the specification of the model as in Example \ref{Example Mean reverting dynamics and linear demand function}, for $\alpha \in \mathcal A$, the function $E_t:= \mathbb E [ X_t ]$ satisfy the ordinary differential equation
$$
dE_t = (\mathbb E [\alpha_t] - \delta E_t)dt, \quad E_0 =m_0. 
$$
Assuming $m_0\geq 0$, since $A \subset [0,\infty)$ we have $\mathbb E [\alpha_t] \geq 0$ so that $E_t \geq0 $ for any $\alpha \in \mathcal A$. 
Thus, at equilibrium we also have $m_t \geq 0$.
Hence, for generic $\xi, \bar \xi \in E$ and  $ m , \bar m$ with $0 \leq\xi \leq \bar \xi$ and $0\leq m_t \leq \bar m _t$, we obtain
\begin{align*}   
    J(\alpha, m;\xi) =& \int_0^T e^{-\rho t} \big( D E_t + \xi m_t E_t - \mathbb E [ (X_t)^2 + \alpha^2_t /2 ] \big) dt \\
    & \quad + e^{-\rho T} \big( D ET + \xi m_T E_T - \mathbb E [ (X_T)^2 ] ) \\
     \leq&  \int_0^T e^{-\rho t} \big( D E_t + \bar \xi \bar m _t E_t - \mathbb E [ (X_t)^2 + \alpha^2_t /2 ] \big) dt \\
     &\quad  + e^{-\rho T} \big( D E_T + \bar \xi \bar m_T E_T - \mathbb E [ (X_T)^2 ] ) \\
    =& J(\alpha,\bar m; \bar \xi),
\end{align*}
which in turn implies
\begin{equation}
     \label{eq J leq J Example 1}
     \inf_{\alpha \in \mathcal A} J(\alpha, m;\xi) \leq \inf_{\alpha \in \mathcal A} J(\alpha, \bar m; \bar \xi).
\end{equation}
From the latter inequality, the same conclusions as in Theorems \ref{theorem comparative statics first at equilibrium} and \ref{theorem comparative statics at equilibrium} hold.

Further, we can derive comparative statics result for the mean equilibrium prices.
In particular, take $\xi, \bar \xi \in E$ and let $\underline m ^\xi, \overline m ^\xi$ and $\underline m ^{\bar \xi}, \overline m ^{\bar \xi}$ the related minimal and maximal MFGE, with associated optimal production $\underline X ^\xi, \overline X ^\xi,  \underline X ^{\bar \xi},$ and $\overline X ^{\bar \xi}$.
From the equilibrium condition $\underline m ^\xi_t = \mathbb E [ \underline X ^\xi_t]$, we can write the related mean equilibrium price as
$$
 \mathbb E [ P(\underline X _t^\xi, \underline m _t^\xi;\xi) ] = D + \xi \underline m ^\xi_t - \underline m ^\xi_t = D + (\xi -1)\underline m ^\xi_t, \quad \text{for any $t \in [0,T]$,}
$$
and similarly for $ \mathbb E [ P(\overline X _t^\xi, \overline m _t^\xi;\xi)] , \mathbb E [ P ( \underline X _t^{\bar \xi}, \underline m _t^{\bar \xi}; \bar \xi)] $ and $ \mathbb E [ P ( \overline X _t^{\bar \xi}, \overline m _t^{\bar \xi}; \bar \xi)]$. 
Therefore, for $\xi \geq 1$, we obtain
$$
\mathbb E [ P(\underline X _t^\xi, \underline m _t^\xi;\xi) ] \leq \mathbb E [ P(\overline X _t^\xi, \overline m _t^\xi;\xi) ], \quad \text{for any $t \in [0,T]$, }
$$
giving a comparative statics of the mean equilibrium prices for fixed $\xi$.
Finally, for $\xi \leq \bar \xi$, with  $\bar \xi \geq 1$, we obtain
$$
\begin{aligned}
   & \mathbb E [ P(\underline X _t^\xi, \underline m _t^\xi;\xi) ] \leq \mathbb E [ P ( \underline X _t^{\bar \xi}, \underline m _t^{\bar \xi}; \bar \xi)]
     \leq \mathbb E [ P ( \overline X _t^{\bar \xi}, \overline m _t^{\bar \xi}; \bar \xi)], 
     \quad \text{for any $t \in [0,T]$,}
\end{aligned}   
$$
and, if also $\xi \geq 1$, we have
$$
\begin{aligned}
   & \mathbb E [ P(\underline X _t^\xi, \underline m _t^\xi;\xi) ] \leq 
    \mathbb E [ P(\overline X _t^\xi, \overline m _t^\xi;\xi) ] \leq \mathbb E [ P ( \overline X _t^{\bar \xi}, \overline m _t^{\bar \xi}; \bar \xi)],
    \quad \text{for any $t \in [0,T]$,}
\end{aligned}
$$
which gives us the desired monotonicity.

\subsubsection{Comparative statics for Examples \ref{Example Mean reverting log-dynamics and isoelastic inverse demand} and \ref{Example Geometric dynamics and isoelastic inverse demand} }
In both examples, the assumptions of Theorem \ref{theorem comparative statics first at equilibrium} are satisfied, hence the maximal equilibrium is the one associated to the maximal reward. 

We now investigate the a comparative statics for the equilibrium prices for Example \ref{Example Mean reverting log-dynamics and isoelastic inverse demand}.
In particular, 
for generic $\xi \in E$,  $ m $  and $X$ such that $m_t = \exp( \mathbb E [\log X_t ] )$, 
taking the logarithm of the price we obtain
\begin{equation*}
    \mathbb E [ \log P( X _t,  m _t;\xi) ] = \log D+ \xi \log \gamma + \xi \log m_t - \zeta \mathbb E [ \log X_t] = \log D  + \xi \log \gamma  + (\xi - \zeta) \log m_t. 
\end{equation*}
Therefore, for $\xi \geq \zeta$, we have monotone log-price; i.e., 
\begin{equation*}
    \mathbb E [ \log P( \underline X ^\xi _t,  \underline m _t;\xi) ] \leq  \mathbb E [ \log P( \overline X ^\xi _t,  \overline m _t;\xi) ], 
\end{equation*}
for the minimal and maximal MFGEs  $\underline m ^\xi $, $ \overline m ^\xi$  with associated optimal productions $\underline X ^\xi$, $ \overline X ^\xi$.

\begin{remark}
    We underline that the comparative statics as in Theorem \ref{theorem comparative statics at equilibrium} hold even if the (scalar) parameter $\xi$ is replaced by a time dependent function or a stochastic process $\xi :[0,T] \times \Omega \to \mathbb R$.
    In particular, one can consider the case in which a social planner or a exogenous agent can choose the parameter $\xi$ (or some of the parameters of a dynamics for $\xi$) in order to optimize some quantities at equilibria.
    This type of problems appear in contract theory and Stakelberg games involving mean field games (see e.g.\ the recent \cite{elie.Hubert.Mastrolia.Possamai.2021mean} or \cite{guo.hu.zhang.2022optimization} and the references therein) and have received an increasing attention in recent years. 
    For example, a social planner could be interested in keeping the maximal mean price  $\mathbb E [ P(\overline X _t^\xi, \overline m _t^\xi;\xi) ]$ close to a certain target, while incurring in some costs for modifying $\xi$.
    Clearly, comparative statics at equilibria can reveal to be extremely useful in this framework.
\end{remark}

\subsection{Algorithms}
We next address the problem of constructing the MFGE.
Given an imput $m^1$, define:
$$
\begin{aligned}
&\text{\textbf{Algorithm 1 [Banach iteration]:} }\quad  \text{$m_t ^{n+1} := \mathbb E [ X_t^{m^n} ]$;} \\
&\text{\textbf{Algorithm 2 [Fictitious Play]:}} \quad \text{ $\hat m ^n_t := \frac1n \sum_{k=1}^n \nu_t^k$, \  $\nu_t ^{n+1} := \mathbb E [   X_t^{\hat m ^n}  ]$ and $\nu^1 := m^1$.}
\end{aligned}
$$

In order to properly initialize the algorithms, we first discuss a priori bounds for the MFGE according to the two cases in Assumption \eqref{assumption a priori estimates}. 
On the one hand, if $A = [0,\bar a]$, we can define two trajectories  $\underline X$ and $\overline X$ as the solutions to the SDEs
\begin{align*}
    d \underline X_t &=  \underline b (\underline X_t) dt + \sigma(\underline X_t) dW_t, \quad \underline X_0 = x_0, \quad \text{with } \underline b (x) := \min_{a \in A} b(x,a), \\
    d \overline X _t &=\overline b (\overline X _t ) dt + \sigma (\overline X_t) dW_t, \quad \overline X_0 = x_0,\quad \text{with } \overline b (x) := \max_{a \in A} b(x,a),
\end{align*}
and via comparison principle we obtain
\begin{equation*}
    \underline X _t \leq X^\alpha_t \leq \overline X _t, \quad \text{for any  $\alpha$.} 
\end{equation*}
Thus, we can define 
\begin{equation}\label{eq definition m min max}
    m_t^{\text{\tiny{Min}}} := 
    \mathbb E [ \underline X_t  ]
    \quad \text{and} \quad 
     m_t^{\text{\tiny{Max}}} := 
    \mathbb E [ \overline X_t ],
\end{equation}
in order to obtain  
\begin{equation}
    \label{equation controlled equation in interval}
     m_t^{\text{\tiny{Min}}} \leq \mathbb E [X^{\alpha}_t] \leq  m_t^{\text{\tiny{Max}}}, \quad \text{for any  $\alpha$.}  
\end{equation}

Finally, we can state the following convergence result.
\begin{theorem}\label{theorem MFGE convergence}
Under Assumptions \eqref{eq assumption supermodularity} and \eqref{assumption a priori estimates}, the following statements hold true:
\begin{enumerate}
    \item\label{theorem MFGE convergence.banach} Convergence of Banach iteration: 
    \begin{enumerate}
        \item If $m^1 = m^{\text{\tiny{Min}}}$, then $m^n_t \leq m_t^{n+1}$ for any $t \in [0,T]$, $n\geq 1$, and $\lim_n m^n_t = \underline m _t$, for any $t \in [0,T]$.
        \item If $m^1 = m^{\text{\tiny{Max}}}$, then $m^n_t \geq m_t^{n+1}$ for any $t \in [0,T]$, $n\geq 1$, and $\lim_n m^n_t = \overline m _t$, for any $t \in [0,T]$.
    \end{enumerate}
    \item\label{theorem MFGE convergence.fictitious} Convergence of Fictitious Play:
    \begin{enumerate}
    \item If $m^1 = m^{\text{\tiny{Min}}}$, then $\hat m ^n_t \leq \hat m _t^{n+1}$ for any $t \in [0,T]$, $n\geq 1$, and $\lim_n \hat m ^n_t = \underline m _t$, for any $t \in [0,T]$.
    \item If $m^1 = m^{\text{\tiny{Max}}}$, then $\hat m ^n_t \geq \hat m _t^{n+1}$ for any $t \in [0,T]$, $n\geq 1$, and $\lim_n \hat m ^n_t = \overline m _t$, for any $t \in [0,T]$.
    \end{enumerate}
\end{enumerate}
\end{theorem}

\begin{remark}
Notice that, even after showing the convergence of the Banach iteration without necessarily relying on a contraction theorem, its practical implementation could give rise to numerical issues. 
Indeed, since the whole convergence hinges on the monotonicity of the iteration, a small numerical error  when computing the updated distribution could destroy this property, thus preventing the convergence.
Analogous issues could arise for the Fictitious Play algorithm.
However, we verify in the next Section an accurate performance of the numerical schemes related to the Banach iteration in the Examples of Subsection \ref{subsection examples}.
\end{remark}

\section{Numerical analysis}
    \label{section numerical analysis}
In this section we illustrate numerical experiments for some of the examples in Subsection \ref{subsection examples}, choosing $\mu_0 := \delta_{x_0}$ (the Dirac's delta at $x_0$) and $x_0 \in \R$ deterministic.
In particular, since the two Examples \ref{example hold mr log geo} and \ref{Example Mean reverting log-dynamics and isoelastic inverse demand} (resp.\ Examples \ref{Example hold GBM and isoelastic inverse demand} and \ref{Example Geometric dynamics and isoelastic inverse demand}) give qualitatively similar numerical results, we will discuss only the numerics for Examples \ref{Example Mean reverting log-dynamics and isoelastic inverse demand} and \ref{Example Geometric dynamics and isoelastic inverse demand}.

The fixed-point iterations in Theorem \ref{theorem MFGE convergence} have a natural numerical counterpart.
More in detail, in each iteration from $n$ to $n+1$ our implementation of such algorithms goes along the following steps: 
\begin{enumerate} 
\item Start with the initial guess $m$. 
\item Solve the maximization problem for the given $m$.
We implement this step using the dynamic programming approach. 
In particular, for the given  $(m_t)_{t \in [0,T]}$ we can write the Hamilton-Jacobi-Bellman (HJB, in short) equation associated to the representative player optimal control problem: 
\begin{equation*}
\rho V^m= \max_{a \in A} \Big(  x P(x,m;\xi)  - c(a) + b(x,a) \partial_x V^m  +\dfrac{1}{2} (\sigma (x))^2  \partial_{xx} V^m + {\partial_t V^m} \Big), 
\end{equation*}
for $(t,x) \in (0,T) \times \R$, with boundary condition $V^m(T,x) =  e^{-\rho T} x P(x,m;\xi) $ for any $x \in \R$. 
Numerically, we solve a discretized HJB equation backward in time on a non-uniform grid over the domain of $X$ using a standard finite-difference method. The numerical algorithm follows \cite{HACT}, who discuss how discretization scheme satisfies the monotonicity condition required for convergence \cite{barles1991convergence}.
For an interior solution, the first order conditions in the variable $a$ allows to find a feedback 
\begin{equation*}
    \hat \alpha^m (t,x)  = \argmax_{a \in A}  \big( b(x,a) \partial_x V^m (t,x) - c(a) \big) .
\end{equation*} 
Upon a verification theorem (to be shown on a case by case basis), $ \hat \alpha^m (t,x)$ is actually the optimal feedback control.
\item  Compute the distribution of $X^{m}$ by using the infinitesimal generator $\mathcal{L}^m_t$. For a one dimensional diffusion as in our case, we have 
\begin{equation*}
    \mathcal{L}^m_t f (t,x) = b(x, \hat \alpha^m (t,x)) \partial_x f (t,x)  + \dfrac{1}{2} (\sigma(x))^2 \partial_{xx} f(t,x), 
\end{equation*}
The evolution of the distribution $g_t$ is given by the associated Kolmogorov Forward equation
\begin{equation*}
    \partial_t g = \mathcal{L}^{m,*}_t g, \quad g(0) = g_0,
\end{equation*}
where $\mathcal{L}^{m,*}_t$ is the adjoint operator of $\mathcal{L}^m_t$, 
\begin{equation*}
    \mathcal{L}^{m,*}_t g = - \partial_x ( b(x, \hat \alpha^m (t,x)) g) + \partial_{xx} ( \dfrac{1}{2} (\sigma(x))^2 g).
\end{equation*}
\item Update $m^{n+1}$ according to algorithm \ref{theorem MFGE convergence.banach} or \ref{theorem MFGE convergence.fictitious} using the distribution $g_t^n$. Thus, set
$$ m_t^{n+1} := \exp \Big(  \int_0^\infty  \log x \,\,d g_t^n (x) \Big), $$
since we are using the equilibrium condition as in \eqref{eq def MFG equilibrium in the geometric case}.
\item Stop if $\sup_{t \in [0,T]} | m_t^{n+1} - m_t^{n} | < \varepsilon$, otherwise go back to (1).    
\end{enumerate}
In the next two subsections, we illustrate the numerical results of this scheme in two benchmark examples (see Examples \ref{Example Mean reverting log-dynamics and isoelastic inverse demand} and \ref{Example Geometric dynamics and isoelastic inverse demand}).

\subsection{Mean reverting log-dynamics and isoelastic inverse demand}
We now go back to Example \ref{Example Mean reverting log-dynamics and isoelastic inverse demand}. 
Recall that the optimization problem writes as
\begin{equation*}
\begin{aligned}
   \text{Maximize} & \quad  J(\alpha,m):= \mathbb{E} \bigg[ \int_0^T e^{-\rho t} \big(  D ( \gamma m_t)^{\xi} (X_t)^{1-\zeta}  - \dfrac{1}{2} \alpha_t ^2 \big) dt +  e^{-\rho T}   D (\gamma m_T)^{\xi} (X_T)^{1-\zeta} \bigg], \\ 
   \text{subject to}  &\quad d X_t = \alpha_t - \delta X_t \log X_t  dt + \sigma X_t dW_t, \quad X_0 =x_0,
\end{aligned}
\end{equation*}
where $D, \gamma>0, \xi>0$ and $\zeta \in (0, 1)$ are parameters, $A = [0, \bar{a}]$,
 and the equilibrium condition is as in \eqref{eq def MFG equilibrium in the geometric case}. 
$\alpha_t$ is an investment that firms can make in order to increase their future production capacity. 
Enterpreneurs can finance the investment with their own funds, but only up to a limit $\bar{a}$ which captures in a reduced form the fact that entrepreneurs have limited resources. 

Since the optimization is over the set $[0,\bar a]$, we have
\begin{equation}\label{eq optimal control example}
    \hat \alpha^m (t,x)  = \argmax_{a \in [0,\bar a]}  \big( a \partial_x V^m (t,x) - \frac12 a^2 \big) = \min \{ \partial_x V^m (t,x), \bar a\},
\end{equation} 
and the HJB equation writes as
\begin{equation*}
\rho V^m= D (\gamma m)^{\xi} x^{1-\zeta}  - \dfrac{1}{2} (\hat{\alpha}^m) ^2 + \mathcal{L}^m_t V^m, \quad V^m_T = e^{-\rho T}  D (\gamma m_T)^{\xi} x_T^{1-\zeta}.
\end{equation*}
Moreover, the evolution of the distribution $g_t$ is given by the associated Kolmogorov Forward equation
\begin{equation*}
    \partial_t g = \mathcal{L}^{m,*}_t g, \quad g(0) = g_0,
\end{equation*}
where $\mathcal{L}^{m,*}_t$ is the adjoint operator of $\mathcal{L}^m_t$
\begin{equation*}
    \mathcal{L}^m_t f = (\hat{\alpha}^m - \delta x \log x )  \partial_x f + \dfrac{1}{2} \sigma^2 x^2  \partial_{xx} f. 
\end{equation*}

We define a grid with $N_x=501$ points for $x$ between $x_{min}, x_{max}$ chosen sufficiently large, and set parameters to the following values: 
\begin{table}[h]
    \centering 
    \begin{tabular}{lrllr}
    \toprule 
    $\rho$ & 0.02 & & $\sigma$ & 1.0 \\
    $D$ & 1.0 & & $\delta$ & 3.0 \\
    $\zeta$ & 0.5 & & $\bar{a}$ & 12.0 \\ 
    $\xi$ & 3.8 & & $T$ & 1.0 \\
    $\gamma$ & 1.2 & & $dt$ & 0.1 \\
    $x_{min}$ & $e^{-15}$ & & $x_{max}$ & $e^{15}$ 
    \end{tabular}
\end{table}

In our baseline calibration, we initialize the measure $g_0$ as the Dirac's delta of the initial condition $x_0:=1$: 
\begin{equation*}
    g_0 (x = x_0) = \begin{cases}
        1 \quad x = 1, \\ 
        0 \quad x \neq 1.
    \end{cases}
\end{equation*}
In other words, we assume that all firms start with the same level of capacity $x_0=1$. 

Given these parameters, we are interested in studying whether different choices for $m^1$ leads to different equilibria. 
We solve the MFG with the numerical scheme associated to the Banach iteration.
We initialize $m^1$ constantly equal to a value $x^i$ of the points in the grid $ \{ x_{min},..., x_{max}\}$ and then we verify the convergence of the resulting sequence to the equilibria, depending on the initial $x^i$. 

\begin{figure}[htbp]
\centering 
\caption{The two equilibria $\underline{m}, \overline{m}$ in Example \ref{Example Mean reverting log-dynamics and isoelastic inverse demand}.}
\includegraphics[width=0.8\textwidth]{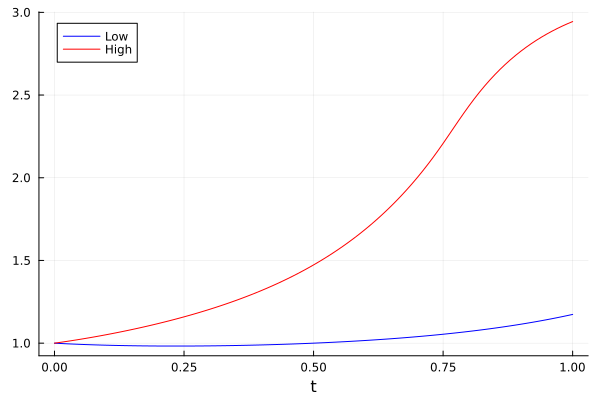}
\label{fig:rev_eqboth}
\end{figure}

\begin{figure}[htbp]
\centering 
\caption{Evolution of the distribution of firms $g_t$ for $t \in [0.05,T]$ in the high (red) and low (blue) equilibria in Example \ref{Example Mean reverting log-dynamics and isoelastic inverse demand}. Color darkness is increasing with $t$.}
\includegraphics[width=0.8\textwidth]{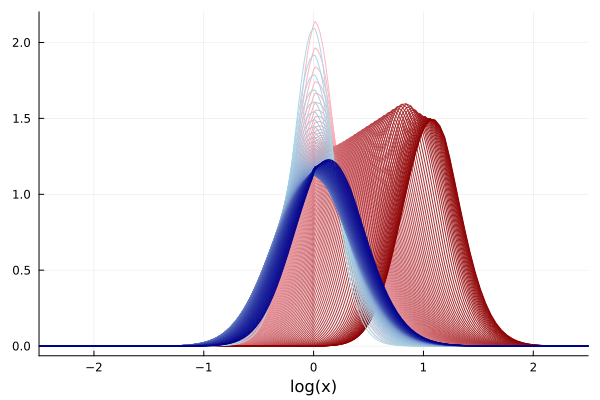}
\label{fig:rev_denspath}
\end{figure}

\begin{figure}[htbp]
\centering 
\caption{The optimal investment $\hat{\alpha}(t,x)$ for $t \in [0,T)$ in the high (red) and low (blue) equilibria in Example \ref{Example Mean reverting log-dynamics and isoelastic inverse demand}. Color darkness is increasing with $t$.}
\begin{subfigure}{0.45\textwidth}
\includegraphics[width=\textwidth]{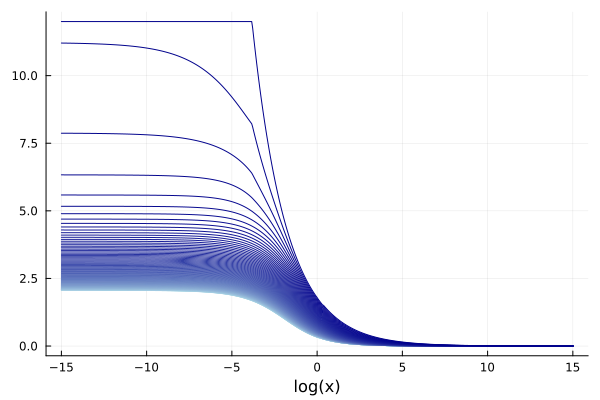}
\end{subfigure}
\begin{subfigure}{0.45\textwidth}
~
\includegraphics[width=\textwidth]{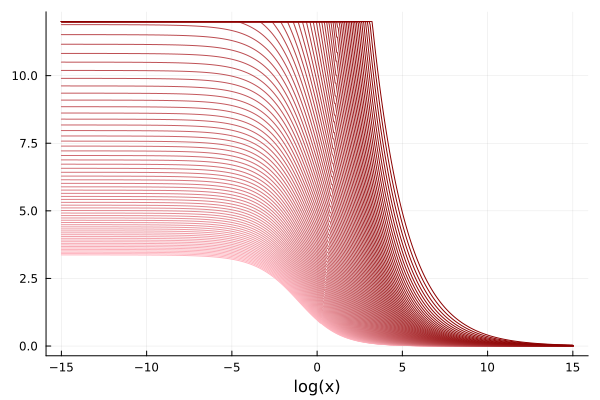}
\end{subfigure}
\label{fig:rev_alphapath}
\end{figure}

Our numerical analysis shows that there are two equilibria, corresponding to high ($\overline m$) and low ($\underline m$) values of $m$, as shown in figure \ref{fig:rev_eqboth}. 
Which equilibrium emerges depends on the initialization of $m^1$.  
If $m^1$ is set high enough, it acts as a "coordination device", pushing firms to increase their investment $\alpha$  as shown in figure \ref{fig:rev_alphapath}.
As a result of the higher investment, in the high equilibrium firms production capacity increases over time and the distribution shifts to the right (see the red lines in Figure \ref{fig:rev_denspath}).
In the low equilibrium instead, investments levels are lower, and so is the resulting capacity distribution. 

\begin{figure}[htbp]
\centering 
\caption{Equilibrium multiplicity as a function of the parameter $\xi$ in Example \ref{Example Mean reverting log-dynamics and isoelastic inverse demand}. The y-axis shows the level of $m^1$ (in logs) above which the model converges to the "high" equilibrium.}
\includegraphics[width=0.8\textwidth]{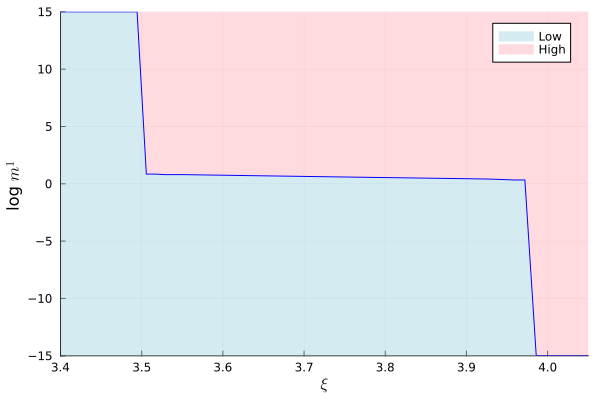}
\label{fig:rev_compxi}
\end{figure}

We perform a comparative static exercise to study how equilibria multiplicity depends on the parameter $\xi$, which regulates the intensity of the strategic interactions among firms. 
Figure \ref{fig:rev_compxi} shows that there is range $[\underline{\xi}, \overline{\xi}]$ in which there are two equilibria. 
On the y-axis we present the values of $\log m^1$ such that the high or low equilibria are selected. 
When $\xi<\underline{\xi}$ or $\xi>\underline{\xi}$ instead there is only one equilibrium and the choice of $m^1$ is irrelevant.

\subsection{Geometric dynamics and isoelastic inverse demand}
We now go back to Example \ref{Example Geometric dynamics and isoelastic inverse demand}. 
Recall that the optimization problem writes as
\begin{equation*}
\begin{aligned}
   \text{Maximize} & \quad  J(\alpha,m):= \mathbb{E} \bigg[\int_0^T e^{-\rho t} \big(  D ( \gamma m_t)^{\xi} (X_t)^{1-\zeta}  - \dfrac{1}{2} \alpha_t ^2 \big) dt +  e^{-\rho T}   D ( \gamma m_T)^{\xi} (X_T)^{1-\zeta} \bigg], \\ 
   \text{subject to}  &\quad d X_t = (\alpha_t - \delta X_t) dt + \sigma X_t dW_t, \quad X_0 =x_0,
\end{aligned}
\end{equation*}
with $A := [0, \bar{a}]$, and the equilibrium condition as in \eqref{eq def MFG equilibrium in the geometric case}.
We leave all parameters unchanged and we again initialize $g_0$ to a degenerate distribution with all mass in $1$.

For $\hat \alpha ^m$ as in \eqref{eq optimal control example}, we can write the HJB equation as
\begin{equation*}
\rho V^m= D (\gamma m)^{\xi} x^{1-\zeta}  - \dfrac{1}{2} (\hat{\alpha}^m) ^2 + \mathcal{L}^m_t V^m, \quad V^m_T = e^{-\rho T}  D (\gamma m_T)^{\xi} x_T^{1-\zeta}.
\end{equation*}
Moreover, the evolution of the distribution $g_t$ is given by the associated Kolmogorov Forward equation
\begin{equation*}
    \partial_t g = \mathcal{L}^*_t g, \quad g(0) = g_0,
\end{equation*}
where $\mathcal{L}^{m,*}_t$ is the adjoint operator of $\mathcal{L}^m_t$
\begin{equation*}
    \mathcal{L}^m_t f = (\hat{\alpha}^m - \delta  x) \partial_x f + \dfrac{1}{2} \sigma^2 x^2  \partial_{xx} f. 
\end{equation*}

\begin{figure}[htbp]
\centering 
\caption{The two equilibria $\underline{m}, \overline{m}$ in Example \ref{Example Geometric dynamics and isoelastic inverse demand}.}
\includegraphics[width=0.8\textwidth]{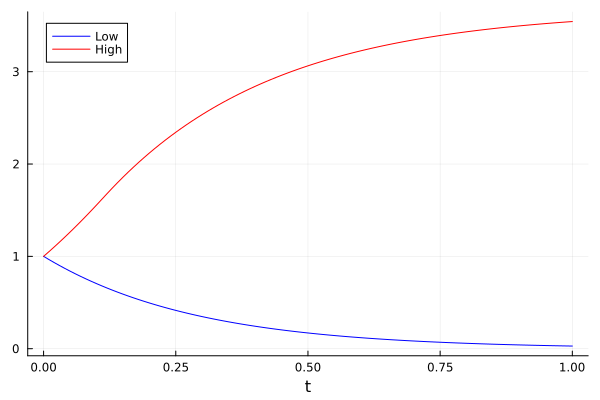}
\label{fig:geo_eqboth}
\end{figure}

\begin{figure}[htbp]
\centering 
\caption{Evolution of the distribution of firms $g_t$ for $t \in [0.05,T]$ in the high (red) and low (blue) equilibria in Example \ref{Example Geometric dynamics and isoelastic inverse demand}. Color darkness is increasing with $t$.}
\includegraphics[width=0.8\textwidth]{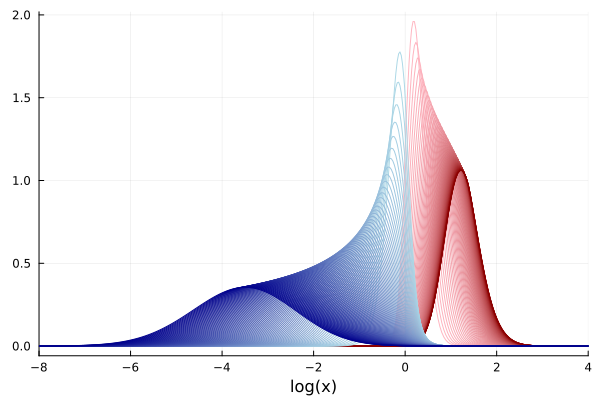}
\label{fig:geo_denspath}
\end{figure}

\begin{figure}[htbp]
\centering 
\caption{The optimal investment $\hat{\alpha}(t,x)$ for $t \in [0,T)$ in the high (red) and low (blue) equilibria in Example \ref{Example Geometric dynamics and isoelastic inverse demand}. Color darkness is increasing with $t$.}
\begin{subfigure}{0.45\textwidth}
\includegraphics[width=\textwidth]{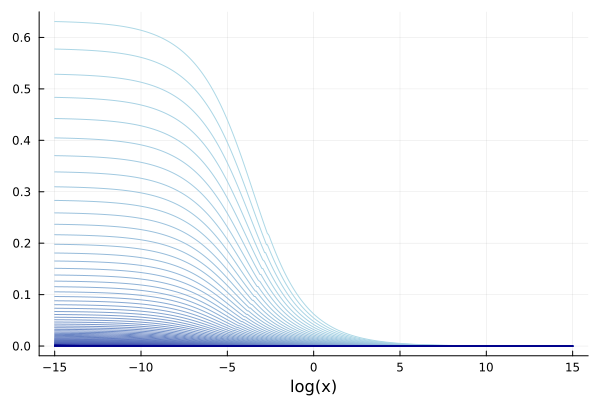}
\end{subfigure}
\begin{subfigure}{0.45\textwidth}
~
\includegraphics[width=\textwidth]{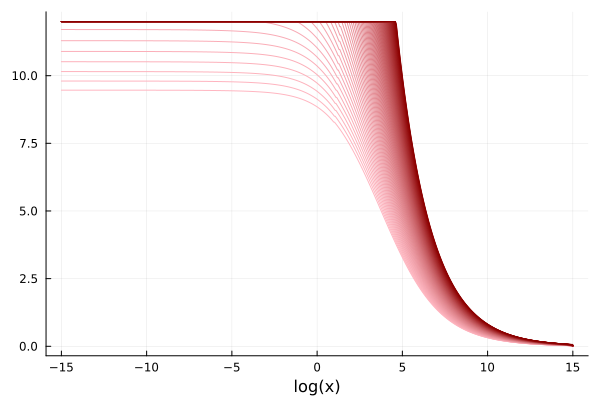}
\end{subfigure}
\label{fig:geo_alphapath}
\end{figure}

\begin{figure}[htbp]
\centering 
\caption{Equilibrium multiplicity as a function of the parameter $\xi$ in Example \ref{Example Geometric dynamics and isoelastic inverse demand}. The y-axis shows the level of $m^1$ (in logs) above which the model converges to the "high" equilibrium.}
\includegraphics[width=0.8\textwidth]{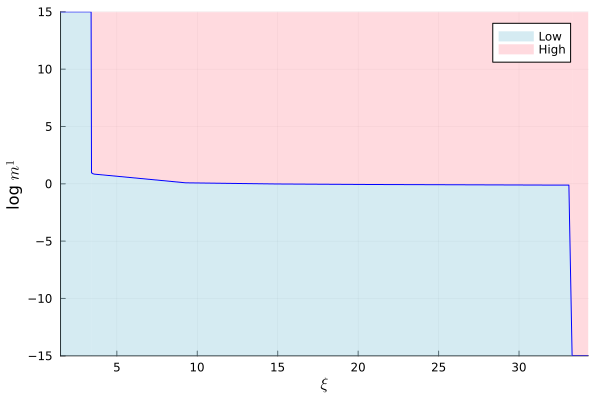}
\label{fig:geo_compxi}
\end{figure}

We perform the numerics and illustrate its result using the same parameters as in the previous subsection. 
Also in this case we find that the model exhibits two equilibria, shown in Figure \ref{fig:geo_eqboth}. 
The absence of the mean reverting component results in a lower minimum equilibrium $\underline m$ and in an higher maximal equilibrium $\overline m$. 
In particular, the minimal equilibrium approximate $0$ towards the end of the game.
Consistently, we see in Figure \ref{fig:geo_denspath} that the minimal and maximal distributions move in opposite directions over time.
Indeed, while the optimal investments at the lower equilibrium declines to $0$ over time, those related to the higher equilibrium rises over time reaching the maximum $\bar a$ for low values of the state (see Figures \ref{fig:geo_alphapath}). 

We perform again a comparative statics in the strength parameter  $\xi$.
Similarly to the previous example, we find a range $[\underline{\xi}, \overline{\xi}]$ in which there are two equilibria (see Figure \ref{fig:geo_compxi}).
However, in comparison to the previous example (cf.\ Figure \ref{fig:rev_compxi}) we notice that the region with multiple equilibria is significantly wider.

\section{Concluding discussion}
    \label{section Concluding discussion}
In this paper, we consider continuous-time mean-field stochastic games with strategic complementarities.
The interaction between the representative productive firm and the population of rivals comes through the price at which the produced good is sold and the intensity of interaction is measured by a so-called "strength parameter" $\xi$. 
Via lattice-theoretic arguments we first prove existence of equilibria and provide comparative statics results when varying $\xi$.  
These give equilibrium-selection criteria depending on the total reward or the price at equilibrium.
Moreover, a careful numerical study based on iterative equilibria converging to suitable maximal and minimal equilibria allows then to study in relevant financial examples how the emergence of  multiple equilibria is related to the strength of the strategic interaction.
This experiments are interesting also from the theoretical point of view, as they illustrate the convergence of the numerical schemes related to the Banach iteration without necessarily relying on a contraction theorem.

Possible extensions of the model could include a common noise, which could be incorporated in a stochastic dynamic strength parameter $(\xi_t)_t$.
The comparative statics results, moreover can be useful when considering the problem of a social planner who can affect the parameter $\xi$ with the intent of optimizing certain quantities related to the minimal or the maximal equilibria $\underline m ^\xi$ and $\overline m ^\xi$.


\appendix

\section{Proofs}
    \label{Appendix}
\subsection{Auxiliary results}
In this appendix we discuss some of the proofs of the results in Section \ref{section A general mean field game model and its properties}.
Thanks to Lemma \ref{lemma optimal controls MFG}, for any $(m _t)_t$ and $\xi$ there exists a unique optimal investment $\alpha^{m,\xi}$ with related production $X^{m,\xi}$.
We can therefore define the best response map
\begin{equation}\label{eq definition brm}
\Lambda^\xi (m) = (\mathbb E [X^{m,\xi}_t ] )_t.
\end{equation}
We first show the following auxiliary result, which is a stronger version of Lemma \ref{lemma brm increasing}.
For $t \in [0,T]$, the symbol $\delta_t$ denotes the Dirac's delta at point $t$; that is, $\delta_t(A)= 1$ if $t \in A$ and $\delta_t(A)= 0$ if $t \notin A$, for any $A$ Borel subset of $[0,T]$. 
\begin{lemma}\label{lemma brm increasing parameters} 
The following statements hold true:
\begin{enumerate}
\item Under condition \eqref{eq assumption supermodularity}, for any $\xi, m , \bar m$ with $m_t \leq \bar m_t$ $dt+\delta_T$-a.e.\ in $[0,T]$, we have $X^{m,\xi}_t \leq X^{\bar m,\xi}_t$ for all $t\in [0,T]$, $\mathbb P$-a.s.\ and $\Lambda ^\xi (m)_t \leq \Lambda^\xi (\bar m )_t$ for all $t\in [0,T]$; 
\item Under conditions \eqref{eq assumption supermodularity} and \eqref{eq assumption supermodularity parameter}, for any $\xi, \bar \xi, m , \bar m$ with $\xi \leq \bar \xi$ and $m_t \leq \bar m_t$ $dt+\delta_T$-a.e.\ in $[0,T]$, we have $X^{m,\xi}_t \leq X^{\bar m, \bar \xi}_t$ for all $t\in [0,T]$, $\mathbb P$-a.s.\ and $\Lambda ^\xi (m)_t \leq \Lambda^{\bar \xi} (\bar m )_t$ for all $t\in [0,T]$.
\end{enumerate}
\end{lemma}
\begin{proof}
We first prove the second claim, thus assuming both conditions \eqref{eq assumption supermodularity} and \eqref{eq assumption supermodularity parameter}.

Take $m, \bar{m}$ such that $ m _t \leq \bar m _t$ and $\xi \leq  \bar \xi$. 
Let $\alpha  = \argmin J ( \cdot, m; \xi)$ and $\bar \alpha  = \argmin J ( \cdot, \bar m; \bar \xi)$ and set $X = X^\alpha$ and $\bar X = X^{\bar \alpha}$.
Following the the proof of Lemma 3.1 in \cite{dianetti2022strong}, one can show that the controls $\alpha^\land, \alpha^\lor$ defined by 
\begin{equation}\label{eq control sup inf} 
\alpha_t^{\land} := \alpha_t \mathds{1}_{ \{ X_t < \bar X_t \}} + \bar{\alpha}_t \mathds{1}_{ \{ X_t \geq  \bar X_t \}} 
\quad  \text{and} \quad 
\alpha_t^{\lor} := \bar{\alpha}_t \mathds{1}_{ \{ X_t < \bar X_t \}} + {\alpha}_t \mathds{1}_{ \{ X_t \geq \bar X_t \}}, 
\end{equation} 
are elements of $\mathcal A$ and such that $X \land \bar X  := (\min \{ X_t,\bar X _t\} ) _{t\in [0,T]} = X^{\alpha^\land}$ and $X \lor \bar X := (\max \{ X_t,\bar X _t\} ) _{t\in [0,T]}= X^{\alpha^\lor}$. 

By the admissibility  of $\alpha^{\land}$ and the optimality of $\alpha$ we can write 
\begin{align*}
0&\leq J(\alpha,{m};\xi) - J(\alpha^{\land},m; \xi ) \\ \notag
&=   \mathbb{E} \bigg[ \int_0^T e^{-\rho t} \big( X_t P(X_t, m_t; \xi)  - X_t \land \bar X_t P(X_t \land \bar X_t, m_t; \xi) -  c(\alpha_t) + c(\alpha^\land _t ) \big)dt  \bigg] \\ \notag
& \quad +  \mathbb{E} \bigg[  e^{-\rho T} \big( X_T P(X_T, m_T; \xi)  - X_T \land \bar X_T P(X_T \land \bar X_T, m_T; \xi) \big)  \bigg] . \notag
\end{align*}
Moreover, by \eqref{eq control sup inf}, we easily find 
$$
\mathbb{E} \bigg[ \int_0^T e^{-\rho t} \big( -  c(\alpha_t) + c(\alpha^\land _t ) \big)dt  \bigg]
=\mathbb{E} \bigg[ \int_0^T e^{-\rho t} \big( -  c(\alpha^\lor_t) + c(\bar \alpha _t ) \big)dt  \bigg], 
$$
which allows to rewrite the previous inequality as
\begin{align*}
0 &\leq \mathbb{E} \bigg[ \int_0^T e^{-\rho t} \big( X_t\lor \bar X_t P( X_t\lor \bar X_t, m_t; \xi)  - \bar X_t P( \bar X_t, m_t; \xi) - c(\alpha^\lor_t) + c(\bar \alpha _t \big)dt  \bigg] \\
& \quad +  \mathbb{E} \bigg[  e^{-\rho T} \big( X_T\lor \bar X_T P( X_T\lor \bar X_T, m_T; \xi)  - \bar X_T P( \bar X_T, m_T; \xi)  \big)  \bigg] . 
\end{align*}
Next, thanks to conditions \eqref{eq assumption supermodularity} and \eqref{eq assumption supermodularity parameter}, we obtain
\begin{align*}
0 &\leq \mathbb{E} \bigg[ \int_0^T e^{-\rho t} \big( X_t\lor \bar X_t P( X_t\lor \bar X_t, \bar m_t; \bar \xi)  - \bar X_t P( \bar X_t, \bar m_t; \bar \xi) - c(\alpha^\lor_t) + c(\bar \alpha _t \big)dt  \bigg] \\
& \quad +  \mathbb{E} \bigg[  e^{-\rho T} \big( X_T\lor \bar X_T P( X_T\lor \bar X_T, \bar m_T; \bar \xi)  - \bar X_T P( \bar X_T, \bar m_T; \bar \xi)  \big)  \bigg] \\
& = J(\alpha^\lor ,\bar{m};\bar \xi) - J(\bar \alpha , \bar m; \bar \xi ), 
\end{align*}
which implies that $\alpha^\lor $ is a maximizer for $J(\cdot,\bar m; \bar \xi)$.
By Lemma \ref{lemma optimal controls MFG}, we conclude that $\alpha^\lor = \bar \alpha$, so that $\bar X = X^{\alpha^\lor} = X \lor \bar X$. 
Thus, we conclude that
$$
X_t \leq \bar X_t \ \text{for any $t\in [0,T]$, $\mathbb P$-a.s.,}
$$
from which we obtain 
$$
\Lambda^\xi (m) = (\mathbb E [ X_t])_t \leq(\mathbb E [\bar X_t])_t = \Lambda^{\bar \xi} (\bar m),
$$
which is the desired monotonicity.

The proof of the first claim simply follows by repeating the argument above for fixed $\xi$, and by observing that in this case condition \eqref{eq assumption supermodularity parameter} is not necessary.
\end{proof}

\subsection{Proof of Theorem \ref{theorem existence MFGE}}
The aim is to employ Tarski fixed point theorem. 
Let $ m^{\text{\tiny{Min}}}$ and $ m^{\text{\tiny{Max}}}$ be as in \eqref{eq definition m min max}. 
Define now the set of functions 
\begin{equation*}
    L := \{ m: [0,T] \to \mathbb R \, |\,  \text{$m$ is measurable and } m^{\text{\tiny{Min}}}_t \leq m_t \leq  m^{\text{\tiny{Max}}}_t, \ \delta_0 + dt+ \delta_T\text{-a.e.\ in }  [0,T] \},
\end{equation*}
and define the order relation $\leq$ by
\begin{equation}
    \label{eq order relation}
    m \leq \bar m \text{ if and only if } m_t  \leq \bar m_t,  \  \delta_0 + dt+ \delta_T \text{-a.e.\ in }  [0,T].
\end{equation}
Notice that, thanks to \eqref{equation controlled equation in interval}, the map $\Lambda^\xi: L \to L$ is well defined.
Observe that the set of MFGE coincides with the set of fixed points of the map $\Lambda$.

The partially ordered set $(L,\leq )$ is a complete lattice and the map $\Lambda$ is nondecreasing (cf.\ Lemma \ref{lemma brm increasing}). 
Thus, by Tarski fixed point theorem, the set of fixed points of $\Lambda$ is a nonempty complete lattice and so is the set of MFGE.
In particular, there exists minimal and maximal MFGE.

\subsection{Proof of Theorem \ref{theorem MFGE convergence}} 
We prove each claim separately.
\subsubsection{Proof of Claim \ref{theorem MFGE convergence.banach}}
We only show the convergence to the minimal equilibrium. The convergence to the maximal equilibrium can be shown with the same arguments.

For $m^1 =  m^{\text{\tiny{Min}}}$, thanks to \eqref{equation controlled equation in interval} one has
\begin{equation*}
    m^{2} = \Lambda^\xi (m^1) \geq m^1.
\end{equation*}
Thus, since $\Lambda^\xi$ is nondecreasing, we have $m^{3} = \Lambda^\xi (m^2) \geq \Lambda^\xi (m^1) = m^{2}$
and by a simple induction argument we find 
\begin{equation}\label{eq monotonicity banach}
\text{$m^n \leq m^{n+1}$ for any $n$.}
\end{equation}
Thus, we can define
\begin{equation*}
    m^*_t = \lim_n m^n_t  = \sup _n m^n_t,  \quad   \delta_0 + dt+ \delta_T \text{-a.e.\ in }  [0,T]
\end{equation*}
and we need to verify that $m^*$ is the minimal MFGE.

We first show that $m^*$  is a MFGE. 
Define $(\alpha^n, X^{n})$ as the optimal pair for $m^{n-1}$. 
By Lemma \ref{lemma brm increasing parameters} and the monotonicity in \eqref{eq monotonicity banach}, we have $X^{n} \leq X^{n+1}$, so that we can define 
\begin{equation*}
    X^{*}_t : = \lim _n X^{n}_t = \sup _n X^{n}_t.
\end{equation*}
Using stability properties and comparison principles together with the stochastic maximum principle (for more details, we refer to \cite{dianetti2022strong}), one can show that $X^*$ is the optimal production for $m^*$.
Moreover, from the latter limits and the monotone convergesnce theorem, we deduce that
$$
m^*_t = \lim_n m^n_t = \lim _n \mathbb E [  X^{n-1}_t ] =  \mathbb E [  X^{*}_t ], 
$$
which show that $m^*$ is a MFGE.
 
We finally show that $m^*$ is the minimal MFGE. 
If $m$ is a MFGE, then we have $m = \Lambda^\xi (m)$, and therefore 
\begin{equation*}
    m^1 \leq m.
\end{equation*}
Thus, by iterating the map $\Lambda^\xi$, we obtain
\begin{equation*}
    m^n \leq m \text{ for any $n$.}
\end{equation*}
Thus, taking limits in $n$ we conclude that 
\begin{equation*}
    m^* \leq m,
\end{equation*}
which complete the proof.

\subsubsection{Proof of Claim \ref{theorem MFGE convergence.fictitious}}
Again, we only show the convergence to the minimal equilibrium. The convergence to the maximal equilibrium can be shown with the same arguments.

For $\hat m ^1 =  m^{\text{\tiny{Min}}}$, thanks to \eqref{equation controlled equation in interval}, one has
\begin{equation*}
    \hat{m}^{2} = \Lambda^\xi (\hat{m}^1) \geq \hat{m}^1.
\end{equation*}
Thus, since $\Lambda^\xi$ is nondecreasing, we have $\Lambda^\xi (\hat{m}^2) \geq \Lambda^\xi (\hat{m}^1 ) = \hat{m}^2$, which in turn gives  
\begin{equation*}    
\hat{ m}^{3} =\frac 1 3 \big( 2 \hat{m}^2 +  \Lambda^\xi (\hat{m}^2) \big) \geq \hat{m}^{2} 
\end{equation*}
and by a simple induction argument we find 
\begin{equation*}
\text{$\hat{m}^n \leq \hat{m}^{n+1}$ for any $n$.}
\end{equation*} 
Thus, we can define $\hat{m}^* = \sup_n \hat{m}^n$ and with the same argument as in the proof of Claim \ref{theorem MFGE convergence.banach}, we can show that $\hat{m}^*$ is the minimal MFGE.

\subsection{Proof of Theorem \ref{theorem comparative statics at equilibrium}}
We use the results of Claim \ref{theorem MFGE convergence.banach} in Theorem \ref{theorem MFGE convergence} and the monotonicity of $\Lambda$ of Lemma \ref{lemma brm increasing parameters}. 

For $\xi \leq \bar \xi$, and $ m ^{1,\xi} := m ^{1,\bar \xi}:=  m^{\text{\tiny{Min}}}$,
define by induction the sequences $m ^{n+1, \xi} := \Lambda^\xi (m ^{n,\xi}) $ and $m ^{n+1, \bar \xi} := \Lambda^{\bar \xi} (m ^{n,\bar \xi})$.
Notice that, by monotonicity of $\Lambda$ in the parameter, we have
\begin{equation*}
    m ^{2,  \xi} = \Lambda^{ \xi} (m ^1) \leq  \Lambda^{\bar \xi} (m ^1)  =  m ^{2, \bar \xi},
\end{equation*}
and that, if $m^{n,\xi} \leq m^{n,\bar \xi}$, then 
\begin{equation*}
     m ^{n+1,  \xi} = \Lambda^{ \xi} (m ^{n,  \xi}) \leq  \Lambda^{\bar \xi} (m ^{n, \bar \xi})  =  m ^{n+1, \bar \xi}.
\end{equation*}
Thus, by induction we obtain that $m^{n,\xi} \leq m^{n,\bar \xi}$ for any $n$, and by  Claim \ref{theorem MFGE convergence.banach} in Theorem \ref{theorem MFGE convergence} we obtain
$$
\underline m ^\xi = \sup_n m^{n,\xi} \leq \sup_n m^{n,\bar \xi} = \underline m^{\bar \xi},
$$
as desired. 

In the same way, we can show that $\overline m ^\xi \leq \overline m^{\bar \xi}$, thus completing the proof.

\smallskip 
\textbf{Acknowledgements.} 
Funded by the Deutsche Forschungsgemeinschaft (DFG, German Research Foundation) - Project-ID 317210226 - SFB 1283.

\textbf{Disclosure statement.} The authors declare that none of them has conflict of interests to mention.

\bibliographystyle{siam}
\bibliography{main.bib}

\end{document}